\documentclass{amsart}

\usepackage{amssymb}

\theoremstyle{plain}

\numberwithin{equation}{section}

\newtheorem*{teorema}{Theorem A}
\newtheorem*{teoremab}{Theorem B}

\newtheorem{lem}[equation]{Lemma}

\newtheorem{prop}[equation]{Proposition}

\newcommand{\trace}{\operatorname{Trace}}

\newcommand{\SL}{\operatorname{SL}}

\newcommand{\PSL}{\operatorname{PSL}}

\theoremstyle{definition}

\newtheorem{rem}[equation]{Remark}
\begin{document}	
\title{On conjugacy classes of $\SL(2,q)$}

\author{Edith Adan-Bante}

\address{Department of Mathematical Science, Northern Illinois University, 
 Watson Hall 320
DeKalb, IL 60115-2888, USA} 

\email{EdithAdan@illinoisalumni.org}

\author{John M. Harris}

\address{Department of Mathematics, University of Southern Mississippi, 
730 East Beach Boulevard, 
Long Beach, MS 39560, USA}

\email{john.m.harris@usm.edu}



\date{2009}
\begin{abstract}
Let $\SL(2,q)$ be the group of $2\times 2$ matrices with determinant one over a finite field $\mathcal{F}$ of size $q$. 
  We prove that if 
$q$ is even, then the product of any two noncentral conjugacy classes of $\SL(2,q)$ is the union of at least
$q-1$ distinct conjugacy classes of $\SL(2,q)$. On the other hand, if $q>3$ is odd, then the product of any two noncentral conjugacy classes of $\SL(2,q)$ is the union of at least
$\frac{q+3}{2}$ distinct conjugacy classes of $\SL(2,q)$. 
 \end{abstract}
\maketitle

\begin{section}{Introduction}
Let $\mathcal{G}$ be a finite group, $A\in \mathcal{G}$  and $A^{\mathcal{G}}=\{A^B\mid B\in \mathcal{G}\}$ 
be the conjugacy class of $A$ in $\mathcal{G}$.
Let $\mathcal{X}$ be a $\mathcal{G}$-invariant subset of $\mathcal{G}$, i.e. 
$\mathcal{X}^A=\{B^A\mid B\in \mathcal{X}\}=\mathcal{X}$ for all $A\in \mathcal{G}$.
  Then $\mathcal{X}$ can be expressed as a union of 
  $n$ distinct conjugacy classes of $\mathcal{G}$, for some integer $n>0$. Set
 $\eta(\mathcal{X})=n$.
 
 Given any conjugacy classes $A^{\mathcal{G}}$, $B^{\mathcal{G}}$ in $\mathcal{G}$, we can check that the 
 product $A^{\mathcal{G}}B^{\mathcal{G}}=\{XY\mid X\in A^{\mathcal{G}}, Y\in B^{\mathcal{G}}\}$ is a $\mathcal{G}$-invariant
 subset  and 
 thus $A^{\mathcal{G}}B^{\mathcal{G}}$ is the union of $\eta(A^{\mathcal{G}}B^{\mathcal{G}})$ distinct conjugacy classes of 
 $\mathcal{G}$. 
 
 Fix a prime $p$ and an integer $m>0$. 
 Let $\mathcal{F} = \mathcal{F}(q)$ be a field with $q=p^m$ elements and $\mathcal{S}=\SL(2,q)=\SL(2,\mathcal{F})$ be the special linear group,
 i.e. the group of $2\times 2$ invertible matrices over $\mathcal{F}$ with determinant $1$.

 It is proved in \cite{symmetric} that for any integer $n>5$, given any nontrivial
 conjugacy classes $\alpha^{S_n}$ and $\beta^{S_n}$ of the symmetric group $S_n$ 
 of $n$ letters, that is $\alpha,\beta\in S_n\setminus \{e\}$,
  if $n$ is a multiple of two or of three, the product 
 $\alpha^{S_n}\beta^{S_n}$ is the union of at least two distinct conjugacy classes,
 i.e. $\eta(\alpha^{S_n}\beta^{S_n})\geq 2$,
 otherwise the product $\alpha^{S_n}\beta^{S_n}$ is the union of at least three 
 distinct conjugacy classes,
 i.e. $\eta(\alpha^{S_n}\beta^{S_n})\geq 3$.
 A similar result is proved for the alternating group
 $A_n$ in \cite{alternating}. 
 
 Fix a prime $p$ and an integer $m>0$. 
 Let $\mathcal{F} = \mathcal{F}(q)$ be a field with $q=p^m$ elements and $\mathcal{S}=\SL(2,q)=\SL(2,\mathcal{F})$ be the special linear group,
 i.e. the group of $2\times 2$ invertible matrices over $\mathcal{F}$ with determinant $1$.
 Given any non-central conjugacy classes $A^{\mathcal{S}}$, $B^{\mathcal{S}}$ in 
 $\mathcal{S}$, is there any relationship between $\eta(A^{\mathcal{S}}B^{\mathcal{S}})$
  and $q$?   
  
  Arad and Herzog conjectured in \cite{arad} that the product of two nontrivial conjugacy classes 
  is never a 
  conjugacy class in a finite nonabelian simple group.
   Thus, when $q\geq 4$ is even we have that 
  $\mathcal{S}=\SL(2,q)=\PSL(2,q)$ is simple and so we must have that $\eta(A^{\mathcal{S}}B^{\mathcal{S}})>1$ unless $A=I$ or
  $B=I$.  In what follows, we expand and refine this statement.
 
 \begin{teorema} Fix a positive integer $m$.
 Let   $A$ and $B$ be
 matrices in $\mathcal{S}=\SL(2, 2^m)$. Then exactly one of the following holds:
 
 (i)  $A^{\mathcal{S}}B^{\mathcal{S}}= (AB)^{\mathcal{S}}$ and at least one of $A$, $B$ is a scalar matrix.
 
   (ii) $A^{\mathcal{S}}B^{\mathcal{S}}$ is the union of at least
  $2^m-1$ distinct conjugacy classes, i.e. $\eta(A^{\mathcal{S}}B^{\mathcal{S}})$ $\geq 2^m-1$. 
 \end{teorema}

 \begin{teoremab}
 Fix an odd prime $p$ and an integer $m>0$ such that $q=p^m>3$.
 Let   $A$ and $B$ be
 matrices in $\mathcal{S}=\SL(2, q)$. Then exactly one of the following holds:
 
 (i)  $A^{\mathcal{S}}B^{\mathcal{S}}= (AB)^{\mathcal{S}}$ and at least one of $A$, $B$ is a scalar matrix.
 
   (ii) $A^{\mathcal{S}}B^{\mathcal{S}}$ is the union of at least
  $\frac{q+3}{2}$ distinct conjugacy classes, i.e. $\eta(A^{\mathcal{S}}B^{\mathcal{S}})\geq \frac{q+3}{2}$.
 \end{teoremab}

 Given any group $G$, 
denote by $\min(G)$ the smallest integer in 
the set $\{\eta(a^G b^G)\mid a,b\in  G\setminus {\bf Z}(G)  \}$. In Proposition \ref{optimaleven},
given any integer $m>0$,
we present matrices $A$, $B$ in $\SL(2,2^m)$  such that $\eta(A^{\SL(2,2^m)}B^{\SL(2,2^m)})=2^m-1$ and thus
Theorem A is optimal. Also, given any $q=p^m>3$, where $p$ is an odd prime and $m$ is a positive integer, 
in Proposition \ref{optimal}
we prove that Theorem B is optimal by presenting matrices where $\min(\SL(2,q))$ is attained.  Also, using 
GAP \cite{GAP4}, we can check that $\min(\SL(2,3))=2$ and thus Theorem B cannot apply when $q=3$.

When $q$ is even, $\SL(2,q)=\PSL(2,q)$ is a simple group of Lie type of characteristic 2. 
Hence, if we require that both $A$ and $B$ are not involutions in Theorem A, the conclusion of Theorem A follows from Theorem 2 of \cite{gow}. We thank Rod Gow for pointing this out to us.

 \end{section}

\begin{section}{Proofs}

{\bf Notation.} We will denote with uppercase letters the matrices and 
with lowercase letters the elements in $\mathcal{F}$. 

\begin{rem}\label{conjugacytypes}
We can describe matrix representatives of conjugacy classes in 
$\mathcal{S}=\SL(2,\mathcal{F})$ by four families or types (\cite{gordon}):

(i) 
$\left(
\begin{array}{cc} r & 0\\
         0 & r\end{array} \right)$ where  $r\in \mathcal{F}$ and $r^2=1$,

(ii) $\left(
\begin{array}{cc} r& 0\\
         0 & s\end{array} \right)$ where  $r,s\in \mathcal{F}$ and $rs=1$.

(iii)$\left(
\begin{array}{cc} s & u\\
         0 & s\end{array} \right)$ where $s\in \mathcal{F}$, $s^2=1$ and $u$ is
         either 1 or a non-square element of $\mathcal{F}$, i.e 
         $u\in\mathcal{F}\setminus\{x^2\mid x\in \mathcal{F}\}$ .

(iv)$\left(
\begin{array}{cc}  0& 1\\
         -1& w\end{array} \right)$, where $w=r+r^q$ and $1=r^{1+q}$ for some $r\in \mathcal{E}\setminus \mathcal{F}$, where 
         $\mathcal{E}$ is a quadratic extension of $\mathcal{F}$.
         \end{rem}
       That is, any conjugacy class $A^{\mathcal{S}}$ of $\mathcal{S}$ must contain
       one of the above matrices. 
  
\begin{rem}\label{argumenttypes}
By Lemma 3 of \cite{symmetric}, we have that $A^{\mathcal{S}}B^{\mathcal{S}}=B^{\mathcal{S}}A^{\mathcal{S}}$.
      Thus if we want to prove that given any non-central conjugacy classes $A^{\mathcal{S}}$ 
         and $B^{\mathcal{S}}$ of $\mathcal{S}$, $\eta(A^{\mathcal{S}}B^{\mathcal{S}})\geq n$ for some integer $n$, it suffices to prove that the statement holds for each of the six combinations of conjugacy 
         classes containing matrices of type
         (ii), (iii) and (iv).
\end{rem}

\begin{rem}\label{tracesargument}
Two matrices in the same conjugacy class have the same trace. Thus, if the matrices
do not have the same trace, then they belong to distinct conjugacy classes. 
\end{rem}
\begin{lem}\label{generalcase}
Let $C=\left(
\begin{array}{cc} a & b\\
         c & d\end{array} \right)\in \mathcal{S}$ and $A=\left(
\begin{array}{cc} e & f\\
         g & h\end{array} \right)\in\mathcal{S}$.
Then 
$$A^{C}=C^{-1}AC=\left(
\begin{array}{cc} a(de-bg)+c(df-bh)& b(de-bg)+d(df-bh)\\
         a(-ce+ag)+c(-cf+ah) & b(-ce+ag)+d(-cf+ah)\end{array} \right),
         $$
         and therefore
  
(i)   $${\left(\begin{array}{cc} r & 0\\
         0 & s\end{array} \right)}^C=\left(
\begin{array}{cc} adr-bcs& bd(r-s)\\
         -ac(r-s) &ads-bcr\end{array} \right).
         $$    
   
(ii)  $${\left(\begin{array}{cc} s& u\\
         0 & s\end{array} \right)}^C=\left(
\begin{array}{cc} s+ucd& ud^2\\
         -uc^2 &s-ucd\end{array} \right).
         $$  
         
(iii) $${\left(\begin{array}{cc} 0& 1\\
         -1 & w\end{array} \right)}^C=\left(
\begin{array}{cc} ab+c(d-bw)& b^2+d^2-bdw\\
         -a^2-c^2+acw
          &-ab+d(-c+aw)\end{array} \right).
         $$    
         
\end{lem}
\begin{proof}
Observe that $C^{-1}=\left(\begin{array}{cc} d& -b\\
         -c & a\end{array} \right)$.
         
        Hence
        \begin{equation*}
           \begin{array}{rcl}
            C^{-1}AC& = &\left(\begin{array}{cc} d& -b\\
         -c & a\end{array} \right)\left(
\begin{array}{cc} e & f\\
         g & h\end{array} \right)  \left(
\begin{array}{cc} a & b\\
         c & d\end{array} \right)\\
         & = &
         \left(
\begin{array}{cc} de-bg & df-bh\\
         -ce+ag & -cf+ah\end{array} \right)
         \left(
\begin{array}{cc} a & b\\
         c & d\end{array} \right)\\
         & = &
         \left(
\begin{array}{cc} a(de-bg)+c(df-bh)& b(de-bg)+d(df-bh)\\
         a(-ce+ag)+c(-cf+ah) & b(-ce+ag)+d(-cf+ah)\end{array} \right).
         \end{array}
         \end{equation*}

\end{proof}

\begin{lem}\label{particularcases}
Let $C=\left(
\begin{array}{cc} a & b\\
         c & d\end{array} \right)\in \mathcal{S}$. Then
         
          (i)
         $\trace({\left(\begin{array}{cc} r & 0\\
         0 & s\end{array} \right)}^C\left(\begin{array}{cc} u & 0\\
         0 & v\end{array} \right))= ad(r-s)(u-v)+(us+vr)$.

          (ii)
         $\trace({\left(\begin{array}{cc} r & 0\\
         0 & s\end{array} \right)}^C\left(\begin{array}{cc} t & u\\
         0 & t\end{array} \right))= t(r+s)-ac(r-s)u$.

          (iii)
         $\trace({\left(\begin{array}{cc} r & 0\\
         0 & s\end{array} \right)}^C\left(\begin{array}{cc} 0 & 1\\
         -1 & w\end{array} \right))=
         (ac+bd)(s-r)+w(ads-bcr)$.

          (iv)
         $\trace({\left(\begin{array}{cc} r & u\\
         0 & r\end{array} \right)}^C\left(\begin{array}{cc} t & w\\
         0 & t \end{array} \right))=2rt-uwc^2$.

          (v)
         $\trace({\left(\begin{array}{cc} r & u\\
         0 & r\end{array} \right)}^C\left(\begin{array}{cc} 0 & 1\\
         -1 & s \end{array} \right))=-ud^2-uc^2+s(r-ucd).$

          (vi)
         $\trace({\left(\begin{array}{cc} 0 & 1\\
         -1& w\end{array} \right)}^C  \left(\begin{array}{cc} 0 & 1\\
         -1 & v \end{array} \right))
          =-a^2-b^2-c^2-d^2+
         bdw+acw+v(-ab+d(-c+aw))$.   
         
 \end{lem}        
 \begin{proof} The result follows by Lemma \ref{generalcase} and 

(i) 
\begin{equation*}
\begin{array}{ll}
\trace({\left(\begin{array}{cc} r & 0\\
         0 & s\end{array} \right)}^C\left(\begin{array}{cc} u & 0\\
         0 & v\end{array} \right)) & =\trace( \left(
\begin{array}{cc} adr-bcs& bd(r-s)\\
         -ac(r-s) &ads-bcr\end{array} \right)\left(\begin{array}{cc} u & 0\\
         0 & v\end{array} \right))\\
&= u(adr-bcs)+v(ads-bcr) \\
&=ad(ur+vs)-bc(us+vr)\\
&=ad(ur+vs)+(1-ad)(us+vr)\\
&=ad(ur+vs-us-vr)+(us+vr)\\
&= ad(u(r-s)-v(r-s))+(us+vr)\\
&=ad(u-v)(r-s)+(us+vr).
\end{array}
\end{equation*}

(ii) 
\begin{equation*}
\begin{array}{ll}
\trace({\left(\begin{array}{cc} r & 0\\
         0 & s\end{array} \right)}^C\left(\begin{array}{cc} t & u\\
         0 & t\end{array} \right))& = \trace(
\left(
\begin{array}{cc} adr-bcs& bd(r-s)\\
         -ac(r-s) &ads-bcr\end{array} \right)\left(\begin{array}{cc} t & u\\
         0 & t\end{array} \right))\\
&=t(adr-bcs)-ac(r-s)u+t(ads-bcr) \\
&=ad(rt+st)-bc(st+rt)-ac(r-s)u\\
&=(ad-bc)t(r+s)-ac(r-s)\\ & =t(r-s)-ac(r-s)u.
\end{array}
\end{equation*}

(iii)
\begin{equation*}
\begin{array}{llll}
\trace({\left(\begin{array}{cc} r & 0\\
         0 & s\end{array} \right)}^C\left(\begin{array}{cc} 0 & 1\\
         -1 & w\end{array} \right)) \\ 
         \;\;\;\;\;\;\;\;\;\;\;\;\;\;\;\;\;\;\;\;\;\;\;=\trace(
\left(
\begin{array}{cc} adr-bcs& bd(r-s)\\
         -ac(r-s) &ads-bcr\end{array} \right)
         \left(\begin{array}{cc} 0 & 1\\
         -1 & w\end{array} \right))\\
\;\;\;\;\;\;\;\;\;\;\;\;\;\;\;\;\;\;\;\;\;\;\;=-bd(r-s)-ac(r-s)+w(ads-bcr)\\
\;\;\;\;\;\;\;\;\;\;\;\;\;\;\;\;\;\;\;\;\;\;\;=(ac+bd)(s-r)+w(ads-bcr).
\end{array}
\end{equation*} 

(iv)
\begin{equation*}
\begin{array}{llll}
\trace({\left(\begin{array}{cc} r & u\\
         0 & r\end{array} \right)}^C\left(\begin{array}{cc} t & w\\
         0 & t \end{array} \right))&= 
         \left(
\begin{array}{cc} r+ucd& ud^2\\
         -uc^2 &r-ucd\end{array} \right)
         \left(\begin{array}{cc} t & w\\
         0 & t \end{array} \right))\\
         &=t(r+ucd)+ w(-uc^2)+t(r-ucd)\\
          & =2rt-uwc^2.
 \end{array}
 \end{equation*}
 
(v)\begin{equation*} 
 \begin{array}{llll}
 \trace({\left(\begin{array}{cc} r & u\\
         0 & r\end{array} \right)}^C\left(\begin{array}{cc} 0 & 1\\
         -1 & s \end{array} \right))&= \trace
    \left(
\begin{array}{cc} r+ucd& ud^2\\
         -uc^2 &r-ucd\end{array} \right)     
\left(\begin{array}{cc} 0 & 1\\
         -1 & s \end{array} \right))\\
         &= -ud^2-uc^2+s(r-ucd)
         \end{array}
         \end{equation*}
         
(vi)
\begin{equation*}
\begin{array}{llcr}
\trace({\left(\begin{array}{cc} 0 & 1\\
         -1& w\end{array} \right)}^C\left(\begin{array}{cc} 0 & 1\\
         -1 & v \end{array} \right))\\
         \;\;\;\;\;\;\;\;\;\;\;\;\;\;\;\;\;\;\;\;\;\;\;= \trace(
         \left(
\begin{array}{cc} ab+c(d-bw)& b^2+d^2-bdw\\
         -a^2-c^2+acw
          &-ab+d(-c+aw)\end{array} \right)
         \left(\begin{array}{cc} 0 & 1\\
         -1 & v \end{array} \right))\\
         \;\;\;\;\;\;\;\;\;\;\;\;\;\;\;\;\;\;\;\;\;\;\;= -b^2-d^2+bdw-a^2-c^2+acw+v(-ab+d(-c+aw))
         \end{array}
         \end{equation*}
\end{proof}
\begin{rem}\label{lineal}
For  any $a, b\in \mathcal{F}$ such that $a\neq 0$, we have that
$\{ax+b\mid x\in \mathcal{F}\}=\mathcal{F}$.
\end{rem} 
\begin{lem}\label{tracescases}
Let $A=\left(\begin{array}{cc} r & 0\\
         0 & s\end{array} \right)$, where $r\neq s$,  and $B$ be any non-central matrix
         in $\mathcal{S}$, i.e. of type (ii), (iii) or (iv) in Remark \ref{conjugacytypes}. 
        Given any $f\in \mathcal{F}$, there exists a matrix $D$ in the product
         $A^{\mathcal{S}} B^{\mathcal{S}}$  such that $\trace(D)=f$. In particular, 
         $\eta(A^{\mathcal{S}}B^{\mathcal{S}})\geq q$ for any non-central matrix 
         $B$. 
\end{lem}

 \begin{proof}
 Given $i\in \mathcal{F}$, 
 set
 $C(i)=\left(
\begin{array}{cc} i & i-1\\
         1 & 1\end{array} \right)$. Observe that $C(i)\in\mathcal{S}$ for all $i\in \mathcal{F}$.
         Fix $u,v$ in $\mathcal{F}$ such that $uv=1$. 
         By Lemma \ref{particularcases} (i), we have that
       \begin{equation}\label{case28}
       \trace(
         {\left(\begin{array}{cc} r & 0\\
         0 & s\end{array} \right)}^{C(i)}\left(\begin{array}{cc} u & 0\\
         0 & v\end{array} \right))=(r-s)(u-v)i+(us+vr).
         \end{equation}

\noindent By Lemma \ref{particularcases} ii) we have that 
\begin{equation}\label{case29}
\trace ({\left(\begin{array}{cc} r & 0\\
         0 & s\end{array} \right)}^{C(i)}\left(\begin{array}{cc} t & u\\
         0 & t\end{array} \right))=-(r-s)ui+t(r+s).
\end{equation}

        Now let 
        $E(i)=\left(
\begin{array}{cc} 1 & i\\
         0 & 1\end{array} \right)$ for $i\in\mathcal{F}$. As before, observe that $E(i)\in \mathcal{S}$ for any 
         $i\in \mathcal{F}$. 
         Then by \ref{particularcases} (iii), we have that
         \begin{equation}\label{case210}
         \trace(
         {\left(\begin{array}{cc} r & 0\\
         0 & s\end{array} \right)}^{E(i)}\left(\begin{array}{cc} 0 & 1\\
         -1 & w\end{array} \right))=(s-r)i+ws.
         \end{equation}
         
         Since $(r-s)(u-v)\neq 0$, $(r-s)u\neq0$, and $s-r\neq 0$, the result follows
         from \eqref{case28}, \eqref{case29},\eqref{case210},  Remark \ref{conjugacytypes}, Remark \ref{tracesargument} and 
         Remark \ref{lineal}.
   \end{proof}

\begin{lem}\label{even} Let $\mathcal{F}$ be a field with $2^m$ elements for some integer $m$ and $a\in \mathcal{F}$ with
$a\neq 0$. 
 Given any $\mathcal{H}\subseteq\mathcal{F}$, the set 
$\{ai^2+c\mid i\in \mathcal{H}\}$ has $|H|$ elements. In particular, if $\mathcal{H}=\mathcal{F}$, then the set $\{ai^2+c\mid i\in \mathcal{H}\}$ has $q$ elements.
\end{lem}
\begin{proof}
Observe that
\begin{equation*} 
|\{ai^2+c\mid i\in \mathcal{H}\}|=|\{ai^2\mid i\in \mathcal{H}\}|=|\{i^2\mid i\in\mathcal{H}\}|.
\end{equation*}
Since $\mathcal{F}$ is a field of characteristic 2, the map $x\mapsto x^2$ is an automorphism of $\mathcal{F}$ and thus $|\{i^2\mid i\in \mathcal{H}\}|=|H|$. 
\end{proof}  

\begin{lem}\label{tracescaseseven}
Let $\mathcal{F}$ be a field with $q=2^m$ elements for some integer $m$. 

(i)  For any $f$ in $\mathcal{F}$, 
 there exists a matrix $D$ in the product ${\left(\begin{array}{cc} 1 & 1\\
         0 & 1\end{array} \right)}^{\mathcal{S}}\left(\begin{array}{cc} 1 & 1\\
         0 & 1 \end{array} \right)^{\mathcal{S}}$ such that           
 $\trace(D)=f$.

(ii)  For any $f$ in $\{i^2+w\mid i \in \mathcal{F}\setminus \{0\}\}
 $, there exists a matrix $D$ in the product ${\left(\begin{array}{cc} 1 & 1\\
         0 & 1\end{array} \right)}^{\mathcal{S}}\left(\begin{array}{cc} 0 & 1\\
         -1 & w \end{array} \right)^{\mathcal{S}}$  such that           
 $\trace(D)=f$.

 (iii) 
 Given any $f\in \mathcal{F}$,
 there exists a matrix $D$ in the product \\ ${\left(\begin{array}{cc} 0 & 1\\
         -1& w\end{array} \right)}^{\mathcal{S}}  \left(\begin{array}{cc} 0 & 1\\
         -1 & v \end{array} \right)^{\mathcal{S}}$ , where $vw\neq0$,   such that           
 $\trace(D)=f$.
 
 Thus given any conjugacy classes $A^{\mathcal{S}},B^{\mathcal{S}}$ in $\mathcal{S}$, where
 $A$, $B$ is either of type (iii) or (iv),  the product
 $A^{\mathcal{S}}B^{\mathcal{S}}$ is the union of at least $q-1$ distinct conjugacy classes. 
\end{lem}
\begin{proof}
(i)
Given any $i\in \mathcal{F}$, set
 $C(i)=\left(
\begin{array}{cc} 1 & 0\\
         i & 1\end{array} \right)$. Observe that $C(i)$ is in $\mathcal{S}$. By
 Lemma \ref{particularcases} (iv)

\begin{equation*}
\trace ({\left(\begin{array}{cc} 1 & 1\\
         0 & 1\end{array} \right)}^{C(i)}\left(\begin{array}{cc} 1 & 1\\
         0 & 1\end{array} \right))=i^2.
\end{equation*} 

By Lemma \ref{even} we have that $\{i^2\mid i\in \mathcal{F}\}=\mathcal{F}$ and thus (i) follows. 

(ii) Given any $i\in \mathcal{F}\setminus \{0\}$, set
$E(i)=\left(
\begin{array}{cc} i^{-1} & 0\\
         0& i\end{array} \right)$. By Lemma \ref{particularcases} (v), we have
       
\begin{equation*}
\trace ({\left(\begin{array}{cc} 1 & 1\\
         0 & 1\end{array} \right)}^{E(i)}\left(\begin{array}{cc} 0 & 1\\
         -1 & w\end{array} \right))=i^2+w.
\end{equation*}

(iii) For $i\in \mathcal{F}$, set
$F(i)=\left(
\begin{array}{cc} i+1 & i\\
         i & i+1\end{array} \right)$. 
        Observe that since $\mathcal{F}$ is of characteristic 2,
we have $\det(F(i))=(i+1)^2-i^2=i^2+1-i^2=1$ and thus $F(i)\in \mathcal{S}$. By Lemma \ref{particularcases}
(vi)  we have

\begin{equation*}
\trace ({\left(\begin{array}{cc} 0 & 1\\
         -1 & w\end{array} \right)}^{F(i)}\left(\begin{array}{cc} 0 & 1\\
         -1 & v\end{array} \right))=vw(i^2+1).
\end{equation*} 

If $vw\neq0$, then the set $\{vw(i^2+1)\mid i\in \mathcal{F}\}=\mathcal{F}$ by Lemma \ref{even}.

Since $\mathcal{F}$ is of characteristic 2, the only matrix representative of type (iii) is 
$\left(\begin{array}{cc} 1 & 1\\
         0 & 1\end{array} \right)$. Also, 
           since $\mathcal{F}$ is of characteristic two, we can check that
         the matrices $\left(
\begin{array}{cc} 1 & 1\\
         0 & 1\end{array} \right)$  and 
         $\left(
\begin{array}{cc} 0 & 1\\
         -1 & 0\end{array} \right)$ are in the same conjugacy class. Thus, with 
         cases (i), (ii), and (iii), we cover all possible combinations of representatives
         of type (iii) and (iv), and by Remark \ref{tracesargument}, the proof of the result is complete. 
         \end{proof}
%\begin{prop}\label{evenfield}
%Let $\mathcal{F}$ be a field with $2^n$ elements for some integer $n$. 
%Then the product of two non-central
%conjugacy classes of $\mathcal{S}=\SL(2,\mathcal{F})$ is the union of at least 
%$q-1$ distinct conjugacy classes. 
%\end{prop}
\begin{proof}[Proof of Theorem A]
If at least one of  $A$ or $B$ is in the center, i.e. $A$ or $B$ are of type (i) in Remark \ref{conjugacytypes}, then $A^{\mathcal{S}} B^{\mathcal{S}}=(AB)^{\mathcal{S}}$. 
Theorem A then follows from Remark \ref{argumenttypes}, Lemma \ref{tracescases} and
Lemma \ref{tracescaseseven}.
\end{proof}

\begin{prop}\label{optimaleven} Fix $q=2^m$ for some integer $m>0$ and let $\mathcal{F}$ be a field with $q$ elements.  
Let $A=\left(\begin{array}{cc} 1 & 1\\
         0 & 1\end{array} \right)$ and $B=\left(\begin{array}{cc} 0 & 1\\
         -1 & w\end{array} \right)$ in $\mathcal{S}$, where $x^2-wx+1$ is an irreducible polynomial over 
         $\mathcal{F}$. Then 
         $\eta(A^{\mathcal{S}}B^{\mathcal{S}})=q-1$.
         
Hence, Theorem A is optimal.
         \end{prop}
         \begin{proof}  
         Set $C=\left(\begin{array}{cc} a & b\\
         c& d\end{array} \right)$ in $\mathcal{S}$. 
   By Lemma \ref{particularcases} (iv), we have that $\trace(A^C B)=-d^2-c^2+w(1-cd)$. Since 
    $x^2-wx+1$ is an irreducible polynomial over 
         $\mathcal{F}$, $-d^2-c^2+w(1-cd)\neq w$ for any $c,d\in \mathcal{F}$. Thus 
         the matrices in $A^{\mathcal{S}}B^{\mathcal{S}}$ do not have trace $w$. Also, since
         the eigenvalues of $B$ are not in $\mathcal{F}$ and the eigenvalues of $A$ is 1, then 
         $A^{\mathcal{S}}\neq (B^{-1})^{\mathcal{S}}$ and so 
         the identity $I$ is not in $A^{\mathcal{S}}B^{\mathcal{S}}$.
         
         Since  $\mathcal{F}$ has even characteristic and 
    $I$ is not in $A^{\mathcal{S}}B^{\mathcal{S}}$, we conclude that there is a one-to-one correspondence
   of the conjugacy classes in $A^{\mathcal{S}}B^{\mathcal{S}}$ with the traces of the matrices:  if the trace is $0$, then the matrix is similar to a matrix of type (iii), and otherwise, the matrix is similar to a matrix of type (ii) or type (iv), depending on whether or not its characteristic equation is reducible. Thus 
   $\eta(A^{\mathcal{S}}B^{\mathcal{S}})=q-1$.
   \end{proof}
\begin{lem}\label{distinctoftypeiii}
  Let $u,v$ in $\mathcal{F}$.
  Then the matrices $\left(
\begin{array}{cc} s & v\\
         0 & s\end{array} \right)$ and $\left(
\begin{array}{cc} s & u\\
         0 & s\end{array} \right)$ are similar if and only if $v=ud^2$ for some $d$ in $\mathcal{F}\setminus\{0\}$.
         In particular,
         if $u$ is a non-square, then the matrices 
         $\left(
\begin{array}{cc} s & 1\\
         0 & s\end{array} \right)$ and $\left(
\begin{array}{cc} s & u\\
         0 & s\end{array} \right)$ are not similar,
          i.e. they belong to distinct conjugacy classes. 
      \end{lem}
      \begin{proof}
      By Lemma \ref{generalcase} (ii), we have that 
      ${\left(\begin{array}{cc} s& u\\
         0 & s\end{array} \right)}^C=\left(
\begin{array}{cc} s+ucd& ud^2\\
         -uc^2 &s-ucd\end{array} \right).
         $  Thus, if $ \left(
\begin{array}{cc} s+ucd& ud^2\\
         -uc^2 &s-ucd\end{array} \right)= \left(
\begin{array}{cc} t& v\\
         0 &t \end{array} \right)$, then $c=0$ and $s=t$ or $u=v=0$ and hence, $v=ud^2$. 
         \end{proof}
         
 \begin{lem}\label{odd}
  Let $\mathcal{F}$ be a finite field with $q$ elements
  and 
 $a,b,c\in \mathcal{F}$ with $a\neq 0$.
If $q$ is an odd number, then the set $\{ai^2+bi+c \mid i\in \mathcal{F}\}$
 has exactly
$\frac{q+1}{2}$ elements. 
\end{lem}
\begin{proof}
Since the field $\mathcal{F}$ is of odd characteristic, we have that $2\neq 0$ and thus 
\begin{equation*}
\begin{array}{llcr}
\mid \{ai^2+bi+c\mid i\in \mathcal{F}\}\mid &=
|\{i^2+\frac{b}{a}i+ \frac{c^2}{a}\mid \mathcal{F}\}|\\
&=|\{i^2+\frac{b}{a}i+ (\frac{b}{2a})^2\mid \mathcal{F}\}|\\
&=|\{(i+\frac{b}{2a})^2\mid i\in \mathcal{F}\}|=|\{i^2\mid i\in \mathcal{F}\}|.
\end{array}
\end{equation*}
Since $q$ is odd, then  $2$ divides $q-1$ and so the square of the set of units forms a subgroup
of order $\frac{q-1}{2}$. Since $0^2=0$, we have that
the set $\{ai^2+bi+c \mid i\in \mathcal{F}\}$
has exactly $\frac{q-1}{2}+1=\frac{q+1}{2}$ elements.
 \end{proof}
\begin{lem}\label{countingelements} Let $F$ be a finite field of size $q$, where $q$ is odd. 

(i) Fix $a, b\in \mathcal{F}\setminus \{0\}$, and suppose that $q>3$.  The set $\{ax^2+by^2\mid x, y\in \mathcal{F}\setminus \{0\}\}$ has a square and a non-square element. 

(ii) Fix $r$, $s$ and $u$ in $\mathcal{F}$ with $s^2-4\neq 0$. 
Then the set $\{-ux^2-uy^2+s(r-uxy)\mid x,y\in \mathcal{F}, (x,y)\neq (0,0)\}$ has at least 
$q-1$ elements. 

(iii) Let $w\in \mathcal{F}$ be such that $w^2-4\neq 0$. Then 
 $\{a^2-c^2+acw\mid a,c\in \mathcal{F}, a\neq 0\}=\mathcal{F}\setminus \{0\}$.  
\end{lem}
\begin{proof}

{\bf (i)} 
Observe that if $c\in \mathcal{F}$ is a non-square element, i.e. $c\not\in \{i^2\mid i \in \mathcal{F}\}$, then 
$\mathcal{F}=\{i^2\mid i \in \mathcal{F}\}\cup  \{ci^2\mid i\in \mathcal{F}\}$.
Thus, if $x$ is a square and $y$ is a non-square element, then $ax$ is a square and $ay$ is a non-square element when $a$ is a square, and otherwise, $ax$ is a non-square element and $ay$ is a square.  
Thus the set
$\{ax^2+by^2\mid x, y\in \mathcal{F}\setminus \{0\}\}$ has a square and a non-square element if and only if $\{x^2+\frac{b}{a} y^2\mid x, y\in \mathcal{F}\setminus \{0\}\}$ has 
a square and a non-square element.  Hence, without loss of generality, we may assume that $a=1$.  

Note that there are 
$\frac{q+1}{2}$ square elements in $\mathcal{F}$ and there are $\frac{q-1}{2}$ non-square elements in
$\mathcal{F}$.  Also, since $|\mathcal{F}|>3$, if $\mathcal{F}$ has characteristic $p\neq 3$ then $3^2+4^2=5^2$, otherwise there exists an element $w\in \mathcal{F}$ such that $w^2+1=0$.  Thus the set  $\{x^2+y^2\mid x,y\in \mathcal{F}\}$ always contains a square element. 
If $b$ is a square element,  the set $\{x^2+by^2\mid x,y\in \mathcal{F}\setminus \{0\}\}=\{x^2+y^2\mid x,y\in \mathcal{F}\}$ has a square and a non-square element, otherwise the set of square elements would be a subfield of $\mathcal{F}$ of size $\frac{q+1}{2}$. 

Suppose that $b$ is a non-square element.  If $-1$ is a non-square element, then $\{x^2+by^2\mid x,y \in F\setminus\{0\}\} = \{x^2-y^2\mid x,y \in F\setminus\{0\}\}$, and hence we may assume that $b=-1$. Let $\epsilon$ be a generator of $\mathcal{F}$. Observe that $\epsilon$ is not a square and if $x=\frac{\epsilon+1}{2}= 1+y$, then $x^2-y^2=(x-y)(x+y)=\epsilon$ and $x,y\in \mathcal{F}\setminus \{0\}$ since $|\mathcal{F}|>3$. Also for any 
   $x=y\neq 0$, we have that $x^2-y^2=0$ and thus the set 
   $\{x^2-y^2\mid x,y\in \mathcal{F}\setminus \{0\}\}$ contains a square, namely zero,
   and a non-square element, namely $\epsilon$.  We may assume then that $-1$ is a square element.  
   
Given $z\in \mathcal{F}\setminus \{0\}$, set $y=xz$. Then $x^2+by^2=x^2(1+bz^2)$. 
Observe that $1+bz^2=0$ does not have a solution since $-1$ is square and so $-1/z^2$ is a square. Since the set
$Z=\{1+bz^2\mid z\in \mathcal{F}\setminus 0\}$ has $\frac{q-1}{2}$ elements and $0,1\not\in Z$, 
either $Z$ has a square and a non-square element, or it has only
non-square elements. In the first case,
 since $ \{x^2(1+bz^2)\mid x, z\in \mathcal{F}\setminus\{ 0\}\}\subset
\{x^2+by^2\mid x,y\in \mathcal{F}\}$, the result follows. We may assume now
that $Z$ is the set of non-square elements. 
Given $w\in \mathcal{F}\setminus \{0\}$, set $x=wy$. Then $x^2+by^2=y^2(w^2+b)$.  Hence
$W=\{w^2+b\mid w\in \mathcal{F}\setminus 0\}$ has $\frac{q-1}{2}$ elements and $b, -b\not\in W$. Since both $b$ and $-b$ are non-squares, it follows that either
$W$ contains both square and non-square elements, or $W$ contains only square elements.
 In the first case, the result follows as before. In the second case, since
$\{x^2(1+bz^2)\mid x, z\in \mathcal{F}\setminus\{0\}\}$ is the set of all non-square elements and 
$\{y^2(w^2+b)\mid y, w\in \mathcal{F}\setminus\{0\}\}$ is the set of all nonzero square elements, (i) follows.

{\bf (ii)} Observe that $-ux^2-uy^2+s(r-uxy)=f$ if and only if 
\begin{equation}\label{basiceq}
x^2+y^2-sxy+\frac{f-sr}{u}=0
\end{equation}
\noindent for some 
$(x,y)$. Observe that \eqref{basiceq} has a solution if there is some $y\in \mathcal{F}$ such that
the discriminant $\Delta(y)= (sy)^2-4(y^2+\frac{f-sr}{u})= (s^2-4)y^2 + \frac{4(f-sr)}{u}$ is a square. 
Since $s^2-4\neq 0$, by Lemma \ref{odd} we have that the set $\{(s^2-4)y^2 + \frac{4(f-sr)}{u}\mid y\in\mathcal{F}\}$ has at least
$\frac{q+1}{2}$ elements. Since there are $\frac{q+1}{2}$ squares, for some $y$ we must have that $\Delta(y)$ is a square. It follows
then that $x=\frac{sy\pm \sqrt{\Delta(y)}}{2}$ is a solution for \eqref{basiceq}. Observe that 
$(x,y)=(0,0)$ is a solution for the equation \eqref{basiceq} if and only if $f-sr=0$. We conclude that for at
least $q-1$ elements of $\mathcal{F}$, \eqref{basiceq} has a solution $(x,y)$ with $(x,y)\neq (0,0)$. 

{\bf (iii)} 
Observe that the equation $x^2+y^2-xyw=0$ has the unique solution $(x,y)=(0,0)$ since
$w^2-4$ is not a square and so 
the discriminant $\delta(y)=y^2w^2-4y^2=y^2(w^2-4)$ is a square if and only if $y=0$. As before, we can check
that for any $f\in \mathcal{F}\setminus\{0\}$, the equation $x^2+y^2-xyw=f$ has a solution with $x\neq0$.
\end{proof}

\begin{lem}\label{oddiii}
Let $q>3$ be odd, and let $\mathcal{F}$ be a finite field with $q$ elements.

(i)  Given any $f$ in $\{2rt-uwi^2\mid i\in \mathcal{F}\}$, 
 there exists a matrix $B$ in the product ${\left(\begin{array}{cc} r & u\\
         0 & r\end{array} \right)}^{\mathcal{S}}\left(\begin{array}{cc} t & w\\
         0 & t \end{array} \right)^{\mathcal{S}}$ such that           
 $\trace(B)=f$. Also for some $a,b\in\mathcal{F}$, where $a$ is a square and $b$ is a non-square,
   the matrices $\left(
\begin{array}{cc} rt & a\\
         0 & rt\end{array} \right)$ and 
        $ \left(
\begin{array}{cc} rt & b\\
         0 & rt\end{array} \right)$ are in 
         the product ${\left(\begin{array}{cc} r & u\\
         0 & r\end{array} \right)}^{\mathcal{S}}\left(\begin{array}{cc} t & w\\
         0 & t \end{array} \right)^{\mathcal{S}}$. Thus
         $\eta( {\left(\begin{array}{cc} r & u\\
         0 & r\end{array} \right)}^{\mathcal{S}}\left(\begin{array}{cc} t & w\\
         0 & t \end{array} \right)^{\mathcal{S}})\geq \frac{q+3}{2}$.

(ii)  For any $f$ in $\{-uc^2-ud^2+s(r-cd)\mid c,d \in \mathcal{F}\setminus \{0\}\}$,
 there exists a matrix $D$ in the product ${\left(\begin{array}{cc} r & u\\
         0 & r\end{array} \right)}^{\mathcal{S}}\left(\begin{array}{cc} 0 & 1\\
         -1 & s \end{array} \right)^{\mathcal{S}}$  such that           
 $\trace(D)=f$. Therefore $\eta( {\left(\begin{array}{cc} r & u\\
         0 & r\end{array} \right)}^{\mathcal{S}}\left(\begin{array}{cc} 0 & 1\\
         -1 & s \end{array} \right)^{\mathcal{S}})\geq q-1$.

We conclude that given any conjugacy classes $A^{\mathcal{S}},B^{\mathcal{S}}$ in $\mathcal{S}$, where
at least one of  $A$ or $B$ is of type (iii),  the product
 $ A^{\mathcal{S}}B^{\mathcal{S}}        $ is the union of at least $\frac{q+3}{2}$ distinct conjugacy classes. 
\end{lem}
\begin{proof}
{\bf (i)}     By Lemma \ref{generalcase}, given any $x,y \in \mathcal{F}\setminus \{0\}$, 
         
      \begin{equation*}\label{calculations}   
     \left(\begin{array}{cc} rt & rwy^2+tux^2\\
         0 & rt\end{array} \right) = \left(\begin{array}{cc} r & ux^2\\
         0 & r\end{array} \right)\left(\begin{array}{cc} t & wy^2\\
         0 & t\end{array} \right)\in  {\left(\begin{array}{cc} r & u\\
         0 & r\end{array} \right)}^{\mathcal{S}}{\left(\begin{array}{cc} t & w\\
         0 & t\end{array} \right)}^{\mathcal{S}} 
         \end{equation*}
         
         By Lemma \ref{countingelements} (i),
          the set $\{rwy^2+tux^2\mid x, y \in\mathcal{F}\setminus \{0\}\}$  has a square and a non-square element. It follows then by Lemma \ref{distinctoftypeiii} that
          there are two matrices that are not similar in the product  with the same trace.

             Given any $i\in \mathcal{F}$, let
 $C(i)=\left(
\begin{array}{cc} 1 & 0\\
         i & 1\end{array} \right)$.    By Lemma \ref{particularcases} (iv), we have that

         \begin{equation*}
         \trace(
         {\left(\begin{array}{cc} r & u\\
         0 & r\end{array} \right)}^{C(i)}\left(\begin{array}{cc} t & w\\
         0 & t\end{array} \right))=2rt-uwi^2.
         \end{equation*}
         
         By Lemma \ref{odd},
          the set $\{2rt-uwi^2\mid i\in \mathcal{F}\}$ has $\frac{q+1}{2}$ elements. Thus  there are at least
          $\frac{q+1}{2}$ distinct values for the traces of the matrices in the product. Since 
            there are  at least two matrices that are not similar in the product  with the same trace
         and  $\frac{q+1}{2}+1=\frac{q+3}{2}$,
          (i) follows.

        {\bf (ii)}
        Let $C=\left(
\begin{array}{cc} a & b\\
         c & d\end{array} \right)\in \mathcal{S}$.
         By Lemma \ref{particularcases} (v), we have
         \begin{equation*} 
         \trace(
         {\left(\begin{array}{cc} r & u\\
         0 & r\end{array} \right)}^{C}\left(\begin{array}{cc} 0 & 1\\
         -1 & s\end{array} \right))=-ud^2-uc^2+s(r-ucd).
         \end{equation*}
         
   By Lemma \ref{countingelements} (ii), we have the set 
   $\{-ud^2-uc^2+s(r-ucd)\mid c,d\in \mathcal{F}, (c,d)\neq (0,0)\}$ has $q-1$ elements and thus 
   (ii) follows. 
\end{proof}

\begin{lem}\label{oddiv}
Let $A= {\left(\begin{array}{cc} 0 & 1\\
         -1 & w\end{array} \right)}$, $B=\left(\begin{array}{cc} 0 & 1\\
         -1 & v \end{array} \right)$,
          where $v,w\in\mathcal{F}$ are such that
         $v^2-4$ and $w^2-4$ are both non-square elements, i.e. $A$ and $B$ are of type (iv). 

 Given any $f$ in $\{ -i^2+i(v-w)+w-2
 \mid i\in \mathcal{F}\}$, 
 there exists a matrix $E$ in the product $A^{\mathcal{S}}B^{\mathcal{S}}$ such that         
 $\trace(E)=f$. 
 
 Let $s$ be a non-square element of $\mathcal{F}$.
 If $v+w\neq 0$, the matrices $\left(
\begin{array}{cc} -1 & 1\\
         0 &-1\end{array} \right)$ and 
        $ \left(
\begin{array}{cc} -1 & s\\
         0 & -1\end{array} \right)$  are in 
         the product $A^{\mathcal{S}}B^{\mathcal{S}}$. Otherwise
     the matrices $\left(
\begin{array}{cc} 1 & 1\\
         0 & 1\end{array} \right)$ and 
        $ \left(
\begin{array}{cc} 1 & s\\
         0 & 1\end{array} \right)$ are in 
         the product $A^{\mathcal{S}}B^{\mathcal{S}}$    
         
          We conclude that given any conjugacy classes $A^{\mathcal{S}},B^{\mathcal{S}}$ in $\mathcal{S}$, where
both  $A$ and $B$ are of type (iv),  the product
 $ A^{\mathcal{S}}B^{\mathcal{S}}        $ is the union of at least $\frac{q+3}{2}$ distinct conjugacy classes. 
         $\eta( A^{\mathcal{S}}B^{\mathcal{S}})\geq \frac{q+3}{2}$.
\end{lem}
\begin{proof}
Given any $i\in \mathcal{F}$, let
 $C(i)=\left(
\begin{array}{cc} 1 &i\\
         0& 1\end{array} \right)$.  Then by Lemma \ref{particularcases} (vi), we have that
         \begin{equation*}
         \trace(A^{C(i)}B)= -i^2+i(v-w)+w-2.
         \end{equation*}
         
Fix $a$ and $c$ in $\mathcal{F}$, where $a\neq 0$. 
Set $t_1=-a^2-c^2+acw$ and $C= \left(\begin{array}{cc} a & \frac{c-aw}{t_1}\\
         c & -\frac{a}{t_1}\end{array} \right)$. We can check that $t_1\neq 0$ since 
         $w^2-4$ is a non-square and $C\in \mathcal{F}$. Also 

  $$ A^C= {\left(\begin{array}{cc} 0 & 1\\
         -1 & w\end{array} \right)}^C= \left(\begin{array}{cc} w&-\frac{1}{t_1}\\t_1&0\end{array}\right).$$
         
    Fix $e$ and $g$ in $\mathcal{F}$, where $e\neq 0$. 
Set $t_2=-e^2-g^2+egv$ and $D= \left(\begin{array}{cc} e & \frac{g}{t_2}\\
         g & \frac{-e+gv}{t_2}\end{array} \right)$. By Lemma \ref{countingelements} (iii), for 
         any $e, g\in \mathcal{F}$ with $e\neq 0$,   $t_2\neq 0$ since 
         $v^2-4$ is a non-square.

  $$ B^D= {\left(\begin{array}{cc} 0 & 1\\
         -1 & v\end{array} \right)}^D= \left(\begin{array}{cc}0&-\frac{1}{t_2}\\t_2& v\end{array}\right).$$   
         Thus 
         
         $$A^C B^D= \left(\begin{array}{cc} -\frac{t_2}{t_1} & -\frac{w}{t_2}-\frac{v}{t_1}\\
         0 & -\frac{t_1}{t_2}  \end{array}\right).$$
        
        Therefore, if $t_1= t_2$ we get that $ A^C B^D=  \left(\begin{array}{cc}-1 & -\frac{w+v}{t_1}\\
         0 & -1  \end{array}\right)$. By Lemma \ref{countingelements} (iii), we have that 
          $\{ a^2+c^2-acw\mid a,c \in \mathcal{F}, a\neq0\}=\mathcal{F}\setminus\{0\}$. Thus
          the set $\{-\frac{w+v}{t_1}\mid t_1= -a^2-c^2+acw, a,c\in \mathcal{F}, a\neq 0    \}$ has 
          $q-1$ elements as long as $w+v\neq 0$. If $w+v=0$, then $w-v\neq 0$ since $w\neq 0$.
          In that case, let $t_1=-t_2$ and thus the set 
          $\{-\frac{w-v}{t_1}\mid t_1= -a^2-c^2+acw, a,c\in \mathcal{F}, a\neq 0    \}$ has $q-1$
          elements. In particular, in both cases the sets contain $1$ and a non-square element and
           the result follows.  
\end{proof}

\begin{proof}[Proof of Theorem B]
If at least one of $A$ and $B$ is a scalar matrix, then $A^{\mathcal{S}}B^{\mathcal{S}}$ is 
a conjugacy class. We may assume then that $A$, $B$ are similar to matrices of type (ii), (iii)  or (iv).  
Theorem B then follows from Lemma \ref{tracescases}, Lemma \ref{oddiii} and Lemma \ref{oddiv}.
\end{proof}

\begin{prop}\label{optimal} Assume that $\mathcal{F}$ is a field of $q$ elements, with $q>3$ odd.

(i) Assume that $q\cong 1 \mod 4$. Let $w$ be a non-square element in $\mathcal{F}$.
Let $A={\left(\begin{array}{cc} 1 & 1\\
         0 & 1\end{array} \right)}$ and $B= \left(\begin{array}{cc} 1 & w\\
         0 & 1\end{array} \right)$. 
 Then 
 \begin{equation}\eta (A^{\mathcal{S}}B^{\mathcal{S}})=\frac{q+3}{2}.
         \end{equation}
         
         (ii) Assume that $q\not\cong 1 \mod 4$. 
Set $E= \left(\begin{array}{cc} 1 & 1\\
         0 & 1\end{array} \right)$. Then
 \begin{equation}\eta (E^{\mathcal{S}}E^{\mathcal{S}})=\frac{q+3}{2}.
         \end{equation}
         
Hence, Theorem B is optimal.
         
\end{prop} 
\begin{proof}
  {\bf (i)} Let $C_i=\left(
\begin{array}{cc} a_i & b_i\\
         c_i & d_i\end{array} \right)\in \mathcal{S}$, for $i=1,2$. 
By Lemma \ref{generalcase} (ii),

   \begin{equation*}
     A^{C_i} B = \left(\begin{array}{cc} 1+c_id_i & {d_i}^2\\
         -{c_i}^2& 1-c_id_i\end{array} \right) \left(\begin{array}{cc} 1 & w\\
         0 & 1\end{array} \right)=\left(\begin{array}{cc} 1 + c_id_i & w+wc_id_i+{d_i}^2\\
         -c_i^2 & 1-c_id_i-w{c_i}^2\end{array} \right).
         \end{equation*}
  
  Hence, $\trace(A^{C_i} B) = 2 - w{c_i}^2$, which takes on $\frac{q+1}{2}$ values by Lemma \ref{odd}.
  Note that if two matrices $A^{C_1} B$ and $A^{C_2} B$ have the same trace, then ${c_1}^2 = {c_2}^2$.

  Suppose ${c_1}^2 = {c_2}^2\neq0$.  Let 
  $D=\left(\begin{array}{cc} 1 & \frac{c_2 d_2 - c_1 d_1}{{c_1}^2} \\ 0 & 1\end{array} \right).$ 
  
  Then $(A^{C_1} B)^D = A^{C_2} B$, and so, excluding trace $2$, each possible value of the trace is obtained by at most one conjugacy class.
         
  Suppose $c_1=c_2=0$.  Then $A^{C_i} B = \left(\begin{array}{cc} 1 & w+{d_i}^2\\
         0 & 1\end{array} \right)$.  Note that $w+{d_i}^2 \neq 0$ since $-w$ is non-square.  Let $z$ be a generator of the multiplicative group of units of $\mathcal{F}$, and suppose $w+{d_i}^2 = z^{n_i}$.  By Lemma \ref{distinctoftypeiii}, $A^{C_1} B$ and $A^{C_2} B$ are conjugate if and only if there is an $e \in \mathcal{F}\setminus\{0\}$ such that $z^{n_1} e^2 = z^{n_2}$, i.e. when $n_1$ and $n_2$ have the same parity.  Hence there are at most two conjugacy classes represented by matrices with trace $2$, and thus, $\eta (A^{\mathcal{S}}B^{\mathcal{S}}) \leq \left(\frac{q+1}{2} - 1\right) + 2 = \frac{q+3}{2}$.

 {\bf (ii)} Define $C_i$ as in (i).  Then we may proceed by replacing $w$ with $1$ in the argument for the previous case, since $-1$ is not a square in $\mathcal{F}$.  Thus, $\eta (E^{\mathcal{S}}E^{\mathcal{S}}) \leq \frac{q+3}{2}$.
 
Hence, in each case, the result follows by Theorem B. 
  \end{proof} 
  \end{section}

\end{document}